\documentclass{amsart}

\usepackage{color}
\newtheorem{theorem}{Theorem}[section]
\newtheorem{lemma}[theorem]{Lemma}

\theoremstyle{definition}
\newtheorem{definition}[theorem]{Definition}
\newtheorem{example}[theorem]{Example}

\newtheorem{hypothesis}[theorem]{Hypothesis}
\newtheorem{corollary}[theorem]{Corollary}
\newtheorem{proposition}[theorem]{Proposition}

\theoremstyle{remark}
\newtheorem{remark}[theorem]{Remark}

\numberwithin{equation}{section}

\newcommand{\bbR}{\mathbb R}
\newcommand{\bbT}{\mathbb T}

\newcommand{\bbZ}{\mathbb Z}

\newcommand{\bbD}{\mathbb D}
\newcommand{\bbC}{\mathbb C}
\newcommand{\bbN}{\mathbb N}

\renewcommand{\epsilon}{\varepsilon}

\newcommand{\be}{\begin{equation}}
\newcommand{\ee}{\end{equation}}

\newcommand{\spec}{\mathrm{spec}}

\newcommand{\cG}{{\mathcal G}}
\newcommand{\cH}{{\mathcal H}}
\newcommand{\cI}{{\mathcal I}}

\newcommand{\cL}{{\mathcal L}}
\newcommand{\cM}{{\mathcal M}}
\newcommand{\cN}{{\mathcal N}}

\newcommand{\cU}{{\mathcal U}}
\newcommand{\cV}{{\mathcal V}}

\newcommand{\KN}{\mathrm{KvN}}

\DeclareMathOperator{\Ker}{\mathrm{Ker}}

\newcommand{\Dom}{\mathrm{Dom}}

\begin{document}

\title [ Canonical commutation relations]
{On dissipative  and non-unitary solutions to operator commutation relations}

\author{K. A. Makarov}
\address{Department of Mathematics, University of Missouri, Columbia, Missouri 63211, USA}
\email{makarov@math.missouri.edu}

\author{E. Tsekanovski\u{i} }
\address{
 Department of Mathematics, Niagara University, PO Box 2044,
NY  14109, USA } \email{tsekanov@niagara.edu}


\dedicatory{Dedicated to the memory of Stanislav Petrovich Merkuriev}


 \maketitle

\begin{abstract} We study the (generalized) semi-Weyl commutation relations
$$
U_gAU_g^*=g(A) \quad \text{ on }\quad  \Dom(A),
$$
where $A$ is a densely defined operator and $G\ni g\mapsto U_g$ is a unitary representation of the subgroup $G$ of the affine group $\cG$, the group of affine transformations of the real axis preserving the orientation.  If $A$ is a symmetric operator, the group $G$ induces an action/flow on  the operator unit ball of contractive transformations from
$\Ker (A^*-iI)$ to $\Ker (A^*+iI)$. We establish several fixed point theorems for this  flow.
In the case of one-parameter continuous subgroups of linear transformations,   self-adjoint  (maximal dissipative) operators associated with the fixed points of the flow  give rise
to solutions of  the (restricted)  generalized Weyl commutation relations. We show that in the dissipative setting,  the restricted Weyl relations admit a variety of non-unitarily equivalent representations.
In the case of deficiency  indices $(1,1)$,   our general results can be strengthened to the level of an alternative.\end{abstract}

\section{Introduction} It is well known (see, e.g.,   \cite{EM} or \cite{Lax}) that the canonical commutation relations in the Weyl form \cite{W}
\begin{equation}\label{niniintro}
U_tV_s=e^{ist}V_{s}U_t, \quad s, t\in \bbR,
\end{equation}
between two strongly continuous unitary groups $U_t=e^{itB}$ and $V_s=e^{isA}$  in a (separable) Hilbert space are satisfied  if and only if
\begin{equation}\label{1/2W}
U_tAU_t^*=A+t I \quad \text{on } \quad \Dom(A), \quad t\in \bbR.
\end{equation}
It is not well known, but trivial  (being  an immediate corollary of the Stone-von Neumann uniqueness result \cite{JvN}), that if a self-adjoint operator $A$ satisfies \eqref{1/2W}, then  $A$  always admits a symmetric restriction $\dot A\subset A$ with deficiency indices (1,1)
such that
\begin{equation}\label{1/2Wdot}
U_t\dot AU_t^*=\dot A+t I \quad \text{on } \quad \Dom(\dot A), \quad t\in \bbR.
\end{equation}

In this setting the following natural question arises. Suppose that \eqref{1/2Wdot} holds for some symmetric operator $\dot A$. More generally, assume that the following commutation relations
\begin{equation}\label{1/2Wdotm}
U_t\dot AU_t^*=g_t(\dot A)\quad \text{on } \quad \Dom(\dot A), \quad t\in \bbR,
\end{equation}
are satisfied, where  $g_t$ is  a one-parameter group  of affine transformations of the real line and  $g_t\mapsto U_t$ is its strongly continuous  representation  by unitary operators.

{\bf The Extension Problem}: Suppose that \eqref{1/2Wdotm} holds for some symmetric operator $\dot A$.
Classify all maximal dissipative, in particular, self-adjoint (if any) solutions  $A $
of the semi-Weyl
 relations
 \begin{equation}\label{mmm}
U_t AU_t^*=g_t(A)\quad \text{on } \quad \Dom( A)
\end{equation}
that extend $\dot A$  such that $\dot A\subset A\subset (\dot A)^*$.

Given the fact that any   one-parameter subgroup of the affine group  is either of the form \begin{equation}\label{g2intro}
g_t(x)=x+vt, \quad t\in \bbR, \quad \text{for some }v\in \bbR,\,\,\,(v\ne 0),
\end{equation}
or
\begin{equation}\label{g1intro}
g_t(x)=a^{ t}(x-\gamma )+\gamma , \quad t\in \bbR, \quad \text{for some} \,\,\,\,0<a\ne 1 \quad  \text{and } \gamma \in \bbR,
\end{equation}
a formal differentiation  of \eqref{mmm}
yields  commutation relations for the generators $B$ and $A$ of the unitary group $U_t=e^{iBt}$ and the  semi-group of contractions $V_s=e^{isA}$, $s\ge 0$, respectively. These relations are of the form
\begin{equation}\label{ins}
[A,B]=ivI,
\end{equation}
in the case of the subgroup $g_t$ of translations \eqref{g2intro},
and
\begin{equation}\label{zwai}
[A, B]=i\lambda A-i\mu I,\quad \lambda=\log a,\quad \mu =\gamma \log a,
\end{equation} for the subgroup  $g_t$ given by \eqref{g1intro}.
Here $[A, B]$ stands for the commutator $[A,B]=AB-BA$ on $\Dom (A)\cap \Dom(B)$.

If we make   the  change of variables:
take $X=A$, $Y=B$ and $Z=ivI$, in the first case, and set $X=A-\gamma I$ and $Y=\frac{1}{i\lambda}B$, in the second, we arrive at the commutation relations
\begin{equation}\label{li1}
[X,Y]=Z, \quad [X,Z]=0, \quad [Y, Z]=0,
\end{equation}
and
 \begin{equation}\label{li2}
[X, Y]=X,
\end{equation}
respectively. Therefore, solving the semi-Weyl relations \eqref{mmm} may be considered a variant of non-self-adjoint quantization of the three- and two-dimensional Lie algebras   \eqref{li1} and \eqref{li2}, respectively. For an alternative approach towards quantization of low-dimensional Lie algebras and their functional models  we refer to \cite{Zol}.

In this paper we solve  a slightly more general extension problem. Namely,  we  assume that
$G\ni g\mapsto U_g$ is a unitary representation of an arbitrary subgroup of the affine group $\cG$, the group of affine transformation of the real axis preserving the orientation.
  We also suppose that $\dot A$ is a densely defined symmetric operator satisfying the semi-Weyl relations
\begin{equation}\label{mm}
U_g\dot AU_g^*=g(\dot A)\quad \text{on } \quad \Dom( \dot A),  \quad g\in G.
\end{equation}
Those $\dot A$'s will be called $G$-invariant operators (with respect to the unitary representation $G\ni g\mapsto U_g$ of the group $G$).
We remark that the $G$-invariant bounded operator colligations (with respect to the group of linear fractional transformations and the  Lorentz group) were studied in \cite{Dub1}, \cite[Ch. 10]{LJ}, \cite{PT}.

Our main results are as follows.
 We show that in the semi-bounded case, that is, if $\dot A$ from   \eqref{1/2Wdotm} is semi-bounded,   the extension problem
 \eqref{mmm}
 is always solvable and that the solution can be given by 
 $G$-invariant self-adjoint operators.
In particular,  if a nonnegative symmetric operator
 is $G$-invariant with respect to the group of all scaling transformations of the real axis into itself $(g(x)=ax$, $ a>0$, $ g\in G )$, then
its  Friedrichs and Krein-von Neumann extensions are
 also 
$G$-invariant. In fact, if the
 symmetric operator has deficiency indices $(1, 1)$, then  those extensions are the only  ones  that are $G$-invariant.

To treat the general case, we study  a flow  of
 transformations on the set $\cV$ of contractive mappings
from the deficiency subspace $\Ker( (\dot A)^*-iI)$ to the
deficiency subspace $\Ker( (\dot A)^*+iI)$ induces by the unitary representation $G\ni g\mapsto U_g$. Based on this study,  we show
 that
$\dot A$ admits a $G$-invariant
maximal dissipative  extension if and only if the flow   has   a fixed
point.

If the   deficiency indices are finite, applying   the Schauder fixed point theorem, we show  that the extension problem
 \eqref{mmm}
is always solvable
in the space of maximal dissipative operators.
If, in addition, the indices are equal,   the flow restricted to the set $\cU\subset \cV$  of all isometries from
$\Ker( (\dot A)^*-iI)$ onto $\Ker( (\dot A)^*+iI)$ leaves this set invariant. In this case,   one can
reduce the search for $G$-invariant self-adjoint solutions  to the extension problem to the one of
fixed points of the restricted  flow. We remark that if $\dot A$ is not semi-bounded, then
self-adjoint $G$-invariant extensions of $\dot A$ may not exist in general,
 even if  the deficiency indices of $\dot A$ are  equal and finite.

Special attention is paid to the case of  deficiency
indices $(1,1)$. In particular,
we  show that if $G$ is a one-parametric continuous
 subgroup of the affine group, the extension problem always has a self-adjoint affine invariant  solution if not with respect to the whole group $G$, but at least
 with respect to a discrete subgroup of $G$.
 This phenomenon,  see Remark \ref{fall}, can be considered an abstract  operator-theoretic counterpart of the
 fall to the center  ``catastrophe" in Quantum  Mechanics \cite{LL}.
 For discussion of the topological origin for this effect see Remark \ref{top}.
 In this connection, 
 we also refer to \cite{Ef} and \cite{FM} for a related discussion of the Efimov Effect in three-body systems and to \cite{FMin}
 where the collapse in a three-body system with  point interactions has been discovered  (also see \cite{MM} and references therein).

\section{G-invariant operators}

Throughout this paper $\cG$ denotes the   ``$ax+b$''-group which is the non-commutative  group of non-degenerate
affine transformations (with respect to composition)  of the real  axis preserving  the
orientation.

Recall that  the group $\cG$ consists of  linear transformations of the real axis
followed by  translations
\begin{equation}\label{affine}
x\mapsto g(x)=ax+b, \quad x\in \bbR,
\end{equation}
with $a>0$ and $b\in \bbR$.

Introduce the concept of  a unitarily affine invariant operator.

Suppose  that  $G$ is  a
 subgroup  of $\cG$.
Assume,  in addition, that  $G\ni g\mapsto U_g$ is a
 unitary representation of $G$ on a separable Hilbert space
$\cH$,
$$
U_fU_g=U_{fg}, \quad f, g \in G.
$$
If $G$ is a continuous group,  assume  that the representation
$G\ni g\mapsto U_g$ is strongly continuous.

\begin{definition}\label{self}
A densely defined closed operator  $ A$ is said to be unitarily affine invariant, or more specifically,  $G$-invariant with respect to a
 unitary representation $G\ni g\mapsto U_g$, if $ \text{ for all } g\in G$
$$
U_g(\Dom (A))= \Dom ( A)
$$
and
\begin{equation}\label{g(T)}
U_gAU_g^*=g(A) \quad \text{ on }\quad  \Dom(A).
\end{equation}
\end{definition}
We notice that  the operator equality \eqref{g(T)} should be
understood in the sense that
\begin{equation}\label{afa}
U_gAU_g^*f=g(A)f=aAf+bf, \quad \text{ for all } f\in \Dom (A),
\end{equation}
whenever the  transformation $g\in G$  is of the form $
x\mapsto g(x)=ax+b $.

The case of one-parameter continuous subgroups of the affine group deserves a special  discussion.

 Recall that if $g_t$ is a one-parameter subgroup of the affine group, then either
the group $G=\{g_t\}_{t\in \bbR}$ consists of
affine transformations of $\bbR$ of the form
\begin{equation}\label{g2}
g_t(x)=x+vt, \quad t\in \bbR,
\end{equation}
for some $v\in \bbR$,  $v\ne 0,$
or
\begin{equation}\label{g1}
g_t(x)=a^{ t}(x-\gamma )+\gamma , \quad t\in \bbR,
\end{equation}
for some $a>0$, $a\ne 1$, and  $ \gamma \in \bbR$.

\begin{remark}\label{suda}
Note that in  case \eqref{g1}, 
in contrast to  \eqref{g2}, 
all  transformations $g_t$, $t\in \bbR$,
have a finite fixed point:
$$
g_t(\gamma)=\gamma \quad \text{ for all } t\in \bbR.
$$
\end{remark}
The concept of $G$-invariant self-adjoint operators in the case
of the group of translations,
$$
g_t(x)=x+t,
$$ is  naturally arises in connection with   the canonical commutation relations in the Weyl form.

\begin{theorem}\label{una}

Suppose that  $G=\{g_t\}_{t\in \bbR}$ is a one-parameter continuous subgroup of the
affine transformations
\begin{equation}\label{shift}
g_t(x)=x+t.
\end{equation}
 Assume that  $A$ is a self-adjoint
  $G$-invariant operator with respect to  a strongly continuous unitary representation $ t\mapsto U_t$
 in a separable  Hilbert space. That is,
 $$
 U_tAU^*_t=A+tI \quad \text{on } \quad \Dom (A).
 $$

Then the unitary group $U_t$ and
 the unitary group $V_s$  generated by $A$,
 $$
V_s=e^{iAs}, \quad s\in \bbR,
$$
 satisfy the Weyl commutation relations
\begin{equation}\label{nini}
U_tV_s=e^{ist}V_{s}U_t, \quad s, t\in \bbR.
\end{equation}

The converse is also true. That is, if two strongly continuous unitary groups $U_t$ and $V_s$  satisfy the Weyl relations \eqref{nini}, then the generator $A$ of the group
$V_s$ is $G$-invariant with respect to the shift group \eqref{shift} and its unitary representation $g_t\mapsto U_t$.
\end{theorem}

As for the proof of this result we refer to \cite[Chapter II, Sec. 7]{Lax}.
\begin{remark} We remark  that    the Stone-von Neumannn uniqueness result \cite{JvN} states that   the shift invariant self-adjoint operator $A$ and
the generator $B$ of the unitary group $U_t$ are mutually  unitarily equivalent to a finite or infinite direct sum of the momentum and position operators from the Schr\"odinger representation:
$$
(A, B)\approx  \bigoplus_{n=1}^\ell (P, Q), \quad \ell=1,2,\dots, \infty.
$$
Here
$$
(Pf)(x)=i\frac{d}{dx}f(x), \quad \Dom (P)=W_2^1(\bbR),
$$
is the momentum operator, and
$$
(Qf)(x)=xf(x), \quad \Dom(Q)=\{f\in L^2(\bbR)\,|\,  xf(\cdot )\in L^2(\bbR) \},
$$
is the position operator, respectively.
\end{remark}

Theorem \ref{una} admits a generalization to the case of arbitrary continuous  one-parameter affine subgroups.
\begin{theorem}\label{dos}

Suppose that  $G$ is a one-parameter continuous subgroup of the affine
group $\cG$. Assume that  $A$ is a self-adjoint
  $G$-invariant operator with respect to a strongly continuous unitary representation $G\ni g_t\mapsto U_t$
 in a separable Hilbert space.

Then the unitary group $U_t$ and
 the  unitary group  $V_s$ generated by $A$,
 $$
V_s=e^{iAs}, \quad s\in \bbR,
$$
 satisfy the  generalized Weyl commutation relations
\begin{equation}\label{ninii}
U_tV_s=e^{isg_t(0)}V_{g'_t(0)s}U_t, \quad s,\,t\in \bbR,
\end{equation}
where
$$
g'_t(0)=\frac{d}{dx}g_t(x)|_{x=0}.
$$

The converse is also true. That is, if two strongly continuous unitary groups $U_t$ and $V_s$  satisfy the generalized Weyl relations \eqref{ninii}, then the generator $A$ of the group
$V_s$ is $G$-invariant with respect to the one-parameter continuous   group $G=\{g_t\}_{t\in\bbR}$ and its unitary representation $G\ni g_t\mapsto U_t$.
\end{theorem}

It is easy to see that if $G=\{g)t\}_{t\in \bbR}$ is the group of translations,
$$
g_t(x)=x+t,
$$
the commutation relations \eqref{ninii} turn into the standard Weyl commutation relations \eqref{nini}.
Indeed, in this case, $g_t(0)=t$ and  $ g'_t(0)=1$,  so, \eqref{ninii} simplifies to  \eqref{nini}. 

For the  further generalizations of the Stone-von Neumann result we refer to \cite{Mac}.

Along with the Weyl  commutation relations \eqref{ninii}, one can also introduce the concept of  restricted  generalized Weyl commutation relations
$$
U_tV_s=e^{isg_t(0)}V_{g'_t(0)s}U_t, \quad t\in \bbR, \,\,\, s\ge0,
$$
involving  a strongly continuous unitary group $U_t$ and a strongly continuous semi-group of contractions $V_s$, $s
\ge 0$. See   \cite{BL,Jorg1,J79,J80,J80a,J81,JM, Sch1, Sch2} where the concept of restricted Weyl commutation relations in the case of the group of affine translations has been discussed.

The following result, which is  an immediate generalization of Theorem \ref{dos} to the case of maximal dissipative $G$-invariant generators,
characterizes  non-unitary  solutions to  the
 restricted  generalized Weyl commutation relations
(see Section 7, Subsections  \ref{can} and \ref{gencan}, for  a number  of  examples of such solutions).


\begin{theorem}\label{bccr2}

Suppose that  $G=\{g_t\}_{t\in \bbR}$ is a one-parameter continuous group of the affine
group $\cG$. Assume that  $A$ is a maximal dissipative
  $G$-invariant operator with respect to unitary representation $G\ni g_t\mapsto U_t$
 in a Hilbert space.
Then the strongly continuous unitary group $U_t$ and
 the strongly continuous semi-group $V_s$ of contractions generated by $A$,
 $$
V_s=e^{iAs}, \quad s\ge0,
$$
 satisfy the restricted
   generalized Weyl commutation relations
\begin{equation}\label{nini2}
U_tV_s=e^{isg_t(0)}V_{g'_t(0)s}U_t, \quad t\in \bbR, \,\,\, s\ge0,
\end{equation}
where
$$
g'_t(0)=\frac{d}{dx}g_t(x)|_{x=0}.
$$

The converse is also true. That is, if  a  strongly continuous unitary group $U_t$ and a  strongly continuous semi-group of contractions $V_s$, $s\ge 0$,   satisfy the generalized restricted
Weyl commutation relations \eqref{nini2}, then the generator $A$ of the semi-group
$V_s$ is $G$-invariant with respect to the one-parameter continuous   group $G=\{g_t\}_{t\in \bbR}$ and its unitary representation $g_t\mapsto U_t$.

\end{theorem}


\begin{proof}
To prove the assertion,
assume that $f\in \Dom (A)$. Then (see, e.g.,  \cite[Theorem 1.3]{KrSel})
$$
V_sf=-\frac{1}{2\pi i}\int_\Gamma e^{i\lambda s} (A-\lambda I)^{-1}f d\lambda, \quad s\ge 0,
$$
where $\Gamma$ is a(ny) contour in the lower half-plane parallel to the real axis and the integral is understood in the principal value sense.
Since $A$ is a $G$-invariant  operator, it is easy to see that
$$
U_t (A-\lambda I)^{-1}f=g'_{-t}(\lambda)(A-g_{-t}(\lambda) I)^{-1}U_t f
$$
and therefore (after a simple change of variable)
\begin{align}
U_tV_sf
&=-\frac{1}{2\pi i}\int_\Gamma e^{i\lambda s} U_t(A-\lambda I)^{-1}f d\lambda \nonumber
\\ &=-\frac{1}{2\pi i}\int_\Gamma e^{i\lambda s}
g'_{-t}(\lambda)(A-g_{-t}(\lambda) I)^{-1}U_t fd\lambda
\nonumber
\\
&=-\frac{1}{2\pi i}\int_{\Gamma'} e^{ig_t(\lambda) s} (A-\lambda I)^{-1}f d\lambda
\nonumber
\\
&=-\frac{1}{2\pi i}\int_{\Gamma'} e^{i(g'_t(0)\lambda+g_t(0)) s} (A-\lambda I)^{-1}f d\lambda
\nonumber
\\
&=e^{isg_t(0)}V_{g'_t(0)s}U_tf, \label{qqqq}
\end{align}
with $\Gamma'=g_{-t}(\Gamma)$, a contour in the lower half-plane.

Thus,  \eqref{qqqq} shows that the representation \eqref{nini2} holds in the strong sense on the dense set   $\Dom (A)$.  Taking into account that   the operators $U_t$, 
$t\in \bbR$,  and $V_s$, $s\ge 0$, are bounded,
one extends \eqref{qqqq} from the dense set to the whole Hilbert space which
proves the claim.

The converse follows by differentiation of the commutation relations.
\end{proof}

\section{$G$-invariant symmetric operators and the extension problem}

The search for self-adjoint or, more generally,  maximal dissipative $G$-invariant operators can be
accomplished solving the  following extension problem:
Given a symmetric $G$-invariant  operator,
 find its all  maximal dissipative   $G$-invariant extensions.
We remark that   the search for $G$-invariant self-adjoint realizations  of the symmetric operator can be undertaken only if the deficiency indices are equal.

We start with the following elementary observation.
\begin{lemma}\label{defpot}
Assume that $\dot A$ is a   $G$-invariant symmetric operator.
Suppose that $A$ is a maximal dissipative extension of $\dot A$.

Then the restriction $A_g$ of the adjoint operator $(\dot A)^*$ onto $D_g=U_g(\Dom(A))$,
\begin{equation}\label{ag}
A_g=(\dot A)^*\vert_{ D_g},
\end{equation}
is a maximal dissipative extension of $\dot A$.
\end{lemma}
\begin{proof}

 Since $\dot A$ is $G$-invariant,
  the  operator
$U_gAU_g^*, $ $ g\in G,
$  is   a maximal dissipative  extension of $g(\dot A)$, and hence
$g^{-1}(U_gAU_g^*)$ is a maximal dissipative extension  of $\dot A$.

Since
$$
\Dom(g^{-1}(U_gAU_g^*))=\Dom(U_gAU_g^*)=U_g(\Dom(A)),
$$
one concludes that the restriction $A_g$  given by \eqref{ag} is a maximal dissipative operator.

\end{proof}

Recall that  the set of all  maximal dissipative extensions of $\dot A$ is in a one-to-one correspondence with  the set $\cV$  of contractive mappings from the deficiency  subspace  $\cN_+=\Ker((\dot A)^*-iI)$ into  the deficiency  subspace $\cN_-=\Ker((\dot A)^*+iI)$.
Note that the set $\cV$ is an  operator  unit ball
 in the  Banach space $\cL( \cN_+, \cN_-)$.

\begin{proposition}\label{vNET}
Let $\dot A$ be a closed symmetric operator and $V\in \cV$ a contractive mapping from
$\cN_+=\Ker((\dot A)^*-iI)$ into $\cN_-=\Ker((\dot A)^*+iI)$. Then the
 restriction $ A$ of the adjoint operator $(\dot A)^*$ on
\begin{equation}\label{domdom}
\Dom(A)=\Dom(\dot A)\dot + (I-V) \Ker ((\dot A)^*-iI)
\end{equation}
 is  a maximal dissipative extension  of $\dot A$.

 Moreover,  the domain of any
maximal dissipative extension $A$ of $\dot A$ such that
$\dot A \subset A\subset (\dot A)^*$ has a decomposition of the form \eqref{domdom},
where $V$ is the restriction of the Cayley transform $(A-iI)(A+iI)^{-1}$
onto the deficiency subspace $\cN_+=\Ker ((\dot A)^*-iI)$.
\end{proposition}

\begin{remark}\label{simetr}
If $\dot A$ has equal deficiency indices  and
$V$ is an  isometric mapping from $\Ker ((\dot A)^*-iI)$ onto $\Ker ((\dot A)^*+iI)$, then the representation   \eqref{domdom},   known as  von Neumann's formula,
provides a parametric description of all self-adjoint extensions of a symmetric operator $\dot A$.
The extension of von Neumann's formulae   to the   dissipative case can be found in \cite{TS} (also see  \cite[Theorem ~4.1.9]{ABT}).
\end{remark}


\section{The flow on the set $\cV$}
 Given a symmetric $G$-invariant operator $\dot A$,
the von Neumann's formula \eqref{domdom}  from  Proposition \ref{vNET} determines a flow  of transformations on the set of  contractive mappings $V\in \cV$. More precisely,
the mapping
$$
G\ni g\mapsto  A_g,
$$
with $A_g$  given by \eqref{ag},
can naturally be ``lifted'' to an  action  of the group $G$
on  the operator  unit ball $\cV$.

In order to describe the action of the group $G$ on $\cV$  in more detail,  suppose that $V\in \cV$ and  let $A$ be  a unique  maximal dissipative extension of $\dot A$
such that
 \begin{equation}\label{flow2}
\Dom (A)=\Dom(\dot A)\dot +(I-V)\Ker( (\dot A)^*-iI).
\end{equation}

Define the extension $A_g$  as in \eqref{ag} by
$$
A_g=(\dot A)^*|_{U_g (\Dom (A))},
$$
and let, in accordance with Lemma \ref{defpot} and  Proposition \ref{vNET}, the  contractive mapping $V_g\in \cV$  be such that
\begin{equation}\label{flow3}
\Dom (A_g)=\Dom(\dot A)\dot +(I-V_g)\Ker( (\dot A)^*-iI),
\quad g\in G.
\end{equation}

Introduce  the action of the group $G$  on the set of a contractive mappings $\cV$ by
\begin{equation}\label{flow1}
\Gamma_g( V)=V_g, \quad g\in G.
\end{equation}

From the definition  of  $\Gamma_g$ it follows
that
$$
\Gamma_{f\cdot g}=\Gamma_f\circ \Gamma_g, \quad f,g\in G,
$$
so that
\begin{equation}\label{ccdd}
\varphi(f\cdot g, V)=\varphi(f, \Phi(g,V)), \quad f, g\in \cG, \quad V\in \cV,
\end{equation}
with
\begin{equation}\label{babyflow}
\varphi(g, V)=\Gamma_g(V).
\end{equation}

Clearly, the maximal dissipative extension $A$ given by \eqref{flow2} is  $G$-invariant if and
only if  $V\in \cV$ is a fixed point of the flow $G\ni g\mapsto
\Gamma_g$. That is,
    \begin{equation}\label{fixp}
 \Gamma_g(V)=V \quad \text{ for  all}\quad  g\in G.
 \end{equation}
Therefore,
the set of  maximal dissipative  extensions of $\dot A$
is in a one-to-one correspondence with the  set of all fixed points of the flow  $\Gamma_g$, and hence, 
 the search for  all $G$-invariant
 maximal dissipative extensions of $\dot A$ can be reduced to the one for the fixed points of the flow \eqref{flow1}.
If, in addition,  the symmetric operator $\dot A$ has equal deficiency indices,  the restriction of the flow $\Gamma_g$ to the set $\cU\subset \cV$ of
all 
 isometric mappings from the deficiency subspace
 $\Ker ((\dot A)^*-iI)$ onto  the deficiency subspace
 $\Ker ((\dot A)^*+iI)$ leaves the set  $\cU$ invariant. In turn,   the set of  all self-adjoint extensions of $\dot A$
is in a one-to-one correspondence with the  set of fixed points of the restricted flow  $\Gamma_g|_{\cU}$.

The following technical result provides a formula representation for the
transformations
 $\Gamma_g$, $g\in \cG$, of the operator unit ball $\cV$ into itself.


\begin{lemma}\label{gammy}
 Assume that $\dot A$ is  a $G$-invariant  closed symmetric densely defined operator  and $\Gamma_g$, $g\in G$, is the flow defined by \eqref{flow2} and \eqref{flow1}.
 Then,
 \begin{equation}\label{final}
\Gamma_g(V)=-[P_-U_g(\gamma I-\delta V)]
[P_+U_g(\alpha I-\beta V)]^{-1},
\end{equation}
 where  $P_\pm$  denotes the orthogonal projections
onto $\Ker ((\dot A)^*\mp iI)$ and $
\alpha= g^{-1}(i)+i$,
$
\beta=g^{-1}(-i)+i
$,
$
\gamma=g^{-1}(i)-i
$,
$
\delta=g^{-1}(-i)-i
$.
 \end{lemma}


\begin{proof}
Given $V\in\cV$, denote by $A$ the maximal dissipative operator obtained
by the restriction of the adjoint operator $(\dot A)^*$ onto the
domain
 $$
\Dom (A)=\Dom(\dot A)\dot +(I-V)\Ker( (\dot A)^*-iI).
$$
Thus, any element  $f\in D(A)$ admits the  representation
\begin{equation}\label{dliaf}
f=f_0\dot +h -Vh,
\end{equation}
where $f_0\in D(\dot A)$, $h\in \Ker ((\dot A)^*-iI)$ and $Vh\in
\Ker ((\dot A)^*+iI)$.

By the definition of the mapping $\Gamma_g$,  $\Gamma_g(V)\in \cV$,  the domain of the operator
$A_g$ given by (cf. \eqref{ag})
$$
A_g=(\dot A)^*\vert_{ U_g(\Dom(A))}.
$$
is of the form
\begin{equation}\label{dag}
\Dom (A_g)=\Dom (\dot A)\dot +(I-\Gamma_g(V))\Ker ((\dot A)^*-iI).
\end{equation}

Suppose that $f\in D(A)$ and that \eqref{dliaf} holds.
Taking into account  that $$U_gf\in
\Dom (U_gAU_g^*)=\Dom(A_g),$$
from  \eqref{dag}  it follows that
\begin{equation}\label{nono}
U_gf=U_gf_0 +U_gh -U_gVh=k_0+m-\Gamma_g(V)m
\end{equation}
for some $k_0\in D(\dot A)$ and some $m\in \Ker ((\dot A)^*-iI)$.
Then
\begin{equation}\label{vagno}
m-\Gamma_g(V)m=U_gh -U_gVh +\ell_0,
\end{equation}
with $\ell_0=U_gf_0-k_0\in D(\dot A)$.
Applying
 $
(\dot A)^*+iI $
 to both sides of \eqref{vagno} one gets
\begin{equation}\label{2im}
2i m=((\dot A)^*+iI)U_gh-((\dot A)^*+iI)U_gVh+(\dot A+iI)\ell_0.
\end{equation}
Here we used that $m\in \Ker ((\dot A)^*-iI)$ and $\Gamma_g(V)m\in
\Ker ((\dot A)^*+iI)$. Since by hypothesis $\dot A$ is
$G$-invariant, the adjoint
operator $(\dot A )^*$ is $G$-invariant as well. One obtains
\begin{equation}\label{ugh}
((\dot A)^*+iI)U_gh=U_g(g^{-1}((\dot A)^* +iI))h=(g^{-1}(i)+i)U_gh.
\end{equation}
Similarly, one gets that
\begin{equation}\label{gvh}
((\dot A)^*+iI)U_gVh=(g^{-1}(-i)+i)U_gVh.
\end{equation}
Combining \eqref{2im}, \eqref{ugh} and \eqref{gvh}, one derives the representation
\begin{equation}\label{2iml}
2i m=(g^{-1}(i)+i)U_gh-(g^{-1}(-i)+i)U_gVh+(\dot A+iI)\ell_0.
\end{equation}

Recall that  $P_\pm$ is the orthogonal projection onto $\Ker
((\dot A)^*\mp iI)$. From \eqref{2iml} it
follows that
\begin{equation}\label{m1}
2i m=P_+U_g\left ( (g^{-1}(i)+i)h- (g^{-1}(-i)+i)Vh\right ).
\end{equation}
Applying $(\dot A)^*-iI$ to  both  sides of \eqref{vagno}, in a
completely analogous way one arrives at the identity
\begin{equation}\label{m2}
-2i\Gamma_g(V) m=P_-U_g((g^{-1}(i)-i)h-(g^{-1}(-i)-i)Vh).
\end{equation}
Combining \eqref{m1} and \eqref{m2} proves the equality
\begin{equation}\label{pochti}
-\Gamma_g(V)P_+U_g(\alpha h- \beta Vh )=P_-U_g(\gamma h-\delta Vh),
\end{equation}
where
$
\alpha= g^{-1}(i)+i$,
$
\beta=g^{-1}(-i)+i
$,
$
\gamma=g^{-1}(i)-i
$, and
$
\delta=g^{-1}(-i)-i
$.

Since
$$
U_g(\Dom(A))=\Dom(U_gAU_g^*),
$$
from  \eqref{nono} and \eqref{m1} it follows that  the operator
$P_+U_g(\alpha I-\beta V)$ maps the deficiency subspace $\Ker((\dot A)^*-iI)$ onto itself.
Moreover, the null space of $P_+U_g(\alpha I-\beta V)$ is trivial.
Indeed, let
$$
P_+U_g(\alpha I-\beta V)h=0
$$
for some $h\in \Ker((\dot A)^*-iI)$ and $m$ be as above
(cf. \eqref{nono}). From \eqref{m1}  it follows that $m=0$.
Therefore, coming back to \eqref{nono}, one concludes that
$$
U_g(\alpha I-\beta V)h\in \Dom(\dot A).
$$
Since $\dot A$ is $G$-invariant, this means that $(\alpha I-\beta V)h\in
\Dom(\dot A)$ which is only possible if $h=0$ ($\alpha$ is never zero). By the closed graph
theorem  the bounded mapping $P_+U_g(\alpha I-\beta V)$ from the deficiency subspace $\Ker((\dot A)^*-iI)$ onto
itself has a bounded inverse and hence from
\eqref{pochti} one concludes that
\begin{equation*}
\Gamma_g(V)=-[P_-U_g(\gamma I-\delta V)][
P_+U_g(\alpha I-\beta V)]^{-1},
\end{equation*}
completing the proof.
\end{proof}


\begin{corollary}\label{homoho} Assume that $G$ is a continuous subgroup of the affine group $\cG$. Assume, in addition, that 
$\dot A$ is a $G$-invariant closed symmetric operator with finite deficiency indices.   Then the mapping
$$(g, V)\mapsto \Gamma_g(V)
$$
from $G\times \cV$ to $\cV$ is continuous. In particular, if $ G=\{g_t\}_{t\in \bbR}$
 is continuous one-parameter subgroup of the affine group $\cG$,
then  for any $V\in \cV$ the trajectory $\bbR\ni t\mapsto
\Gamma_{g_t}(V)$ is continuous.
\end{corollary}

\begin{remark}To validate the usage of the term ``flow" in our  general considerations, recall that  a flow on a closed, oriented manifold $\cM$ is a one-parameter group of homomorphisms of $\cM$. That is, a flow is a function
$$
\varphi:\bbR\times \cM\mapsto \cM
$$
which is continuous and satisfies
\begin{itemize}
\item[(i)] $ \varphi(t_1+t_2,x)=\varphi(t_1, \varphi(t_2, x))$,
\item[(ii)] $\varphi(0,x)=x$
\end{itemize}
for all $t_1,  t_2\in \bbR$ and $x\in \cM$.
If $G=\{g_t\}_{t\in \bbR}$ is a continuous one-parameter group of affine transformations and the deficiency indices of the symmetric operator $\dot A$ are finite, then
by Corollary  \ref{homoho}, relation
\eqref{babyflow} defines a flow, in the standard sense, on the $n$-manifold $\cV$, where
$$n=\dim(\Ker((\dot A)^*-iI))\times\dim(\Ker((\dot A)^*+iI)).$$
\end{remark}


\section{Fixed point theorems}

The main goal of this section is to present several fixed point theorems for   the  flow $\Gamma_g$  given by \eqref{flow1}
on the operator unit ball $\cV$  introduced in the previous section.

We start with recalling the following  general  fixed point theorem  obtained in \cite{MT}.
This theorem solves the problem of the existence of fixed points for the flow \eqref{flow1} in the case where the symmetric operator $\dot A$ is semi-bounded.

To formulate the result we need some preliminaries.

 Recall that if $\dot A$ is a densely defined (closed) positive operator, then the
 set  of all positive self-adjoint extensions of $\dot A$ has the minimal $A_{\KN}$
 (Krein--von Neumann extension) and
 the maximal $A_F$ (Friedrichs extension) elements. That is, for any positive self-adjoint extension
 $\widetilde A$ the following two-sided  operator inequality holds  \cite{Kr1}
 $$
 (A_F+\lambda I)^{-1}\le (\widetilde A+\lambda I)^{-1}\le (A_{\KN}+\lambda I)^{-1}, \quad
 \text{ for all } \lambda >0.
 $$

For the convenience of the reader, recall the following fundamental result characterizing the Friedrichs and
 Krein--von Neumann extensions \cite{AN}.


 \begin{proposition} \label{ando}(\cite{AN})
 Let $\dot A$ be a closed densely defined positive symmetric operator. Denote by $\bf a$ the
 closure of the quadratic form
 \begin{equation}\label{kvf}
 {\dot {\bf a}}[f]=(\dot A
f, f), \quad \Dom[{\dot {\bf a}}]=\Dom(\dot A).
 \end{equation}
 Then the Friedrichs extension $A_F$ coincides with
 the restriction of the adjoint operator
 $(\dot A)^*$ on the domain
 $$
 \Dom (A_F)=\Dom ((\dot A)^*)\cap \Dom [{\bf a}].
 $$
 The Krein--von Neumann extension $A_{\KN}$ coincides with
 the restriction of the adjoint operator
 $(\dot A)^*$ on the domain
  $
 \Dom (A_{\KN})
 $
 which consists of the set of elements $f$ for which there exists a
 sequence $\{f_n\}_{n\in \bbN}$, $f_n\in \Dom(\dot A)$, such that
 $$
 \lim_{n,m\to \infty}{\bf a}[f_n-f_m]=0
 $$
 and
 $$
 \lim_{n\to \infty}\dot Af_n=(\dot A)^*f.
 $$
\end{proposition}

Now we are ready to present the aforementioned  fixed point theorem in the case where the underlying   $G$-invariant symmetric operator is semi-bounded.

\begin{theorem}\label{FK}  Let $G$ be a subgroup of the affine group $\cG$.  Suppose that $\dot A$ is a
closed, densely defined,  semi-bounded from below
 symmetric  $G$-invariant operator with the greatest lower bound  $\gamma$. Then both  the Friedrichs extension $A_F$ of  the operator $\dot
 A$ and  the extension $A_K=(\dot A-\gamma I)_{\KN}+\gamma I$,  where
 $(\dot A-\gamma I)_{\KN}$ denotes
 the Krein--von Neumann extension of the positive symmetric operator
  $\dot A-\gamma I$,  are $G$-invariant.
  \end{theorem}

We provide a short proof the idea of which  is due to Gerald Teschl \cite{GeT} (see \cite{MT} for the original reasoning).
\begin{proof}
Suppose that $\dot A$ is $G$-invariant with respect to a unitary representation $G\ni g\mapsto U_g$.
Denote by   $\cI$   the operator interval of self-adjoint extensions of $\dot A$ with the greatest  lower bound greater or equal  to $\gamma$.
 Recall that  for any $A\in \cI$ one has the operator inequality $A_K\le A\le  A_F $.
 Since the symmetric operator $\dot A$  is $G$-invariant, the correspondence
$$
A \to g^{-1} (U_g A U_g^*), \quad \quad g(x)=ax+b,
$$
with $A$ a self-adjoint extension of $\dot A$,
maps the operator interval $\cI=[A_K, A_F] $ onto itself.
Moreover, this mapping  is operator monotone, that is, it preserves the order.
Therefore, the end-points  $A_K$ and $A_F$  of the operator interval $\cI$ has to be fixed points of this map which completes the proof.
\end{proof}

\begin{remark}\label{rem53} We remark that in the situation of Theorem \ref{FK},    the greatest lower bound $\gamma$ has to be a fixed point for 
all transformations from  the group $G$. Therefore, the hypothesis that $\dot A$ is semi-bounded is rather restrictive. In particular, this hypothesis implies   that  the group   $G$ is necessarily a proper subgroup of the whole affine group $\cG$. For instance, the group $G$ does not contain the transformations $g(x)=x+b$, $b\ne 0,$ without (finite) fixed points.
\end{remark}

Next, we claim (under mild assumptions)  that  for  finite deficiency indices of $\dot A$
 the flow \eqref{flow1}  always has  a fixed point. Therefore,
 any $G$-invariant symmetric operator with finite deficiency indices possesses
 either  a self-adjoint or a  maximal dissipative $G$-invariant extension of $\dot A$ (cf. \cite[Theorem 15]{J80}).

\begin{theorem}\label{Shauder}
Let $G$ be either  a cyclic  or a one-parameter continuous subgroup of affine
transformations of the real axis into itself preserving the orientation.
Suppose that   $G\ni g\mapsto U_g$ is a unitary representation   in a Hilbert space
$\cH$.
Assume that  $\dot A$ is a
 closed symmetric  $G$-invariant operator  with finite deficiency indices.

 Then the  operator $\dot A$
admits  a $G$-invariant maximal dissipative extension.

\end{theorem}

\begin{proof}
For any $g\in G$, the mapping $\Gamma_g: \cV\to \cV$ is a continuous mapping
 from the unit ball $\cV$ into itself. Since the deficiency indices of $\dot A$ are finite,
 $\cV$ is a compact convex set and therefore,  by the Schauder fixed point theorem (see, e.g.,  \cite{LuS}, page 291), the mapping $\Gamma_g$ has a fixed point
 $V\in \cV$. That is,
 $$
 \Gamma_g(V)=V.
 $$
 In particular,
 $$
 \Gamma_{g^n}(V)=V \quad \text{ for all } n \in \bbZ.
 $$
 Therefore, the restriction $A$ of the adjoint operator $(\dot A)^*$ onto the domain
 $$
 \Dom (A)=\Dom (\dot A)+(I-V)\Ker((\dot A)^*-iI)
 $$
 is a $G[g]$-invariant operator, where $G[g]$ is the cyclic group generated by the element $g\in G$, proving the claim in the case of cyclic groups $G$ .

In order to treat the case of one-parameter continuous subgroups $G=\{g_t\}_{t\in \bbR}$,  we remark that
by the first part of the proof for any $ n\in \bbN$   there exists a $V_n\in \cV$
 such that
\begin{equation}\label{sec0}
\Gamma_{g_q}(V_n)=V_n \text{ for all } q\in \bbZ/n.
 \end{equation}
Since by hypothesis $\dot A$ has finite deficiency indices and hence  the unit ball $\cV$  is compact,
one can find  a $V\in \cV$ and a subsequence $\{n_k\}_{k\in \bbN}$ such that
\begin{equation}\label{sec1}
\lim_{k\to \infty}V_{n_k}=V.
\end{equation}

Given $t\in \bbR$,  choose   a sequence of rationals $\{q_k\}_{k\in \bbN}$
such  that
\begin{equation}\label{sec2}
\lim_{k\to \infty}q_k=t
, \quad g_k\in \bbZ/{n_k} .
 \end{equation}
Combining \eqref{sec0} and  \eqref{sec1},
one obtains that
$$
\Gamma_{g_{q_k}}(V_{n_k})=V_{n_k}.
$$
 Going to the limit $k \to \infty$ in this equality,
from \eqref{sec1} and \eqref{sec2} one concludes that
$$
\Gamma_{g_t}(V)=V \quad \text{ for all }  t\in \bbR
$$
upon applying Corollary \ref{homoho}. The proof is  complete.
\end{proof}

\begin{remark}\label{pervaia}  If $\dot A$ has equal deficiency indices and if a fixed point $V$  is
 an isometry from  $\Ker((\dot A)^*-iI)$ onto
$\Ker((\dot A)^*+iI)$, then the corresponding maximal dissipative extension is self-adjoint.
 We also remark that  Theorem \ref{Shauder} does not guarantee the  existence
of a $G$-invariant self-adjoint extensions of $\dot A$ in general. In other words, the restricted flow $\Gamma_g|_{\cU}$,
with $\cU$ the set of all   isometries from  $\Ker((\dot A)^*-iI)$ onto
$\Ker((\dot A)^*+iI)$,  may not have fixed points at all
(see Section 7 for the corresponding examples and Remark \ref{fall} for a detailed discussion of the group-theoretic descent of   those examples).

\end{remark}

\begin{remark}    It is worth mentioning
 that Theorem \ref{Shauder} is a simple consequence of the following Lefschetz fixed point theorem for flows on manifolds.
\begin{proposition}(see, eg., \cite[Theorem 6.28]{Vick})\label{Lef}
If $\cM$ is a closed oriented manifold such that the Euler characteristic $\chi (\cM)$ of $\cM$ is not zero, then any flow on $\cM$ has a fixed point.
\end{proposition}

Indeed, take $\cM=\cV$ and notice that $\cV$ is a closed  convex set with  $\chi(\cV)\ne 0$.

\end{remark}

 %


\section{The case of deficiency indices $(1,1)$}

In this subsection we provide  several results towards the solution of the extension problem \eqref{mmm}
in the dissipative as well as in the self-adjoint settings that are  specific to the case
of deficiency indices $(1,1)$ only.

\begin{hypothesis}\label{kiki}
Assume that $G$ is a  subgroup of the affine group $\cG$. Suppose, in addition, that
$\dot A$ is
 a $G$-invariant closed symmetric operator
  with deficiency indices $(1,1)$.
\end{hypothesis}

Under Hypothesis \ref{kiki},
based on  the results of Proposition  \ref{vNET} and Lemma \ref{gammy} one can identify the flow $G\ni g\mapsto \Gamma_g$ with the representation of  the group $G$ into the group  $\text{Aut}(\bbD)$
of automorphisms of the
unit disk $\bbD$ (consisting of linear-fraction transformations
$z\mapsto e^{i\theta}\frac{z-a}{1-\bar a z}$ for some $|a|<1 $ and $\theta \in \bbR$,
 mapping the unit circle $\bbT$ into itself).

Recall the following   elementary  result.


\begin{proposition}\label{nado}  Any automorphism $\Gamma$ of the (closed) unit disk $\bbD$ different from the identical map
 has at most two  different fixed points in $\bbD$.
 If $\Gamma$ has exactly two fixed points, then both of them lie on the boundary $\bbT$ of $\bbD$.

\end{proposition}



\begin{theorem}\label{tri}
Assume Hypothesis \ref{kiki}.
  Suppose that  $\dot A$ has either
  \begin{itemize}
  \item[(i)] at least three different self-adjoint $G$-invariant
  extensions
  \item[]
  or
  \item[(ii)] at least two $G$-invariant  maximal dissipative extensions one of which is not self-adjoint.
  \end{itemize}

  Then any maximal dissipative extension of $\dot A$ is $G$-invariant. Moreover, in this case,
  if the group  $G$ is nontrivial, then $\dot A$ is not semi-bounded.
\end{theorem}


\begin{proof}
In case (i),   for any $g\in G$ the automorphism $\Gamma_g$  of the unit disk $\bbD$
has at least three different fixed points. Hence, by Proposition \ref{nado},  $ \Gamma_g $ is
the identical map on $\bbD$ and thus  any maximal dissipative extension $A$
 of
$\dot A$  such that $\dot A \subset A\subset (\dot A)^*$ is $G$-invariant.

In case (ii),  for any $g\in G$  the  automorphism   $ \Gamma_g $ has at least two fixed points, one of which is not on the boundary of the unit disk and hence, again, by Proposition \ref{nado},
the  automorphism $ \Gamma_g  $, $g\in G$, is the identical map, proving that
any maximal dissipative extension of $\dot A$ is $G$-invariant.

To prove the last statement of the theorem, assume, to the
contrary, that $\dot A$ is  a semi-bounded symmetric operator.
Suppose, for definiteness, that $\dot A$ is semi-bounded from
below.

Assume that one can find a point $\lambda_0\in \bbR$ to the left
from the greatest lower bound of $\dot A$ such that the set  $\Sigma=\bigcup_{g\in
G}g(\lambda_0)$ is unbounded from below, that is,
\begin{equation}\label{Si}
\inf \Sigma  =-\infty.
\end{equation}
By Krein's theorem (see \cite {Kr1}, Theorem 23) there exists a
self-adjoint extension $A$ of the symmetric operators $\dot A$
such that $\lambda_0$ is a simple eigenvalue of $A$. Since any maximal dissipative extension of $\dot A$ is  $G$-invariant, the self-adjoint operator   $A$ is
also $G$-invariant. Since the spectrum of a $G$-invariant
operator is a $G$-invariant set, one concludes that the set $\Sigma$
 belongs to the pure point spectrum
of $A$ and hence  $A$ is not semi-bounded from below for 
\eqref{Si} holds. The latter is  impossible: any self-adjoint
extension of a semi-bounded symmetric operator with finite
deficiency indices is semi-bounded from below.

To justify the choice of $\lambda_0$ with the property \eqref{Si} one
can argue as follows.

By hypothesis, the group   $G$  is nontrivial and, therefore, $G$
contains a transformations  $g$ either of the form $g(x)=x+b$ for
some $b\ne 0$ or $ g(x)=a(x-\gamma)+\gamma $ for some $a>0$, $a\ne 1$, and
$ \gamma \in \bbR$.

In the first case, the set $\Sigma $ contains the orbit $
\bigcup_{n\in \bbZ} \{ \lambda_0+nb\}$ and hence one can take any
$\lambda_0$ to the left from the greatest lower bound of $\dot A$
to meet the requirement \eqref{Si}. In the second case,
$$
\bigcup_{n\in \bbZ}\{a^n(\lambda_0-\gamma)+\gamma\}\subset  \bigcup_{g\in
G}g(\lambda_0)= \Sigma.
$$
Hence, one can choose any $\lambda_0<\gamma$ to the left from the
greatest lower bound of $\dot A$
 for \eqref{Si} to hold.

\end{proof}

Our next result shows that   a semi-bounded $G$-invariant symmetric operator with deficiency indices $(1,1)$ possesses  self-adjoint $G$-invariant extensions only.


\begin{theorem}\label{FKnet} Assume Hypothesis \ref{kiki}.
  Suppose, in addition,
 that $\dot A$  is semi-bounded from below.

  If $A$ is a maximal dissipative $G$-invariant extension of $\dot A$, then $A$ is necessarily self-adjoint.
 Moreover, either
 $A=A_F$ or  $A=A_K$  where   $A_F$ and $A_K$ are  the operators referred to in Theorem \ref{FK}.
 \end{theorem}


\begin{proof} Assume to the contrary that $ A$, with $\dot A \subset A\subset (\dot A)^*$,
 is  a $G$-invariant maximal dissipative  extension of $\dot A$
different from $A_F$ and $A_K$.

If $A_F\ne A_K$,  then $\dot A$
has at least three different $G$-invariant extensions.
Therefore,  by Theorem
\ref{tri}, the operator $\dot A$ is not semi-bounded from below which contradicts the   hypothesis.  
Thus, either
 $A=A_F$ or  $A=A_K$.

If the extensions       $A_F$ and  $A_K$ coincide, that is,
\begin{equation}\label{coin}
A_F=A_K,
\end{equation}
one proceeds as follows.

By hypothesis,  $A\ne A_F=A_K$,
 the  $G$-invariant maximal dissipative  extension  $A$ is necessarily self-adjoint. 
Indeed, otherwise,  $\dot A$ has
 at least two $G$-invariant  maximal dissipative extensions, $A_F$ and $A$,  one of which is not self-adjoint.  Applying
  Theorem \ref{tri} shows that  the symmetric  operator  $\dot A$ is not semi-bounded which contradicts the hypothesis. Thus,
 $A=A_F=A_K$.

 Denote by $\widetilde \gamma$
 the greatest lower bound of $ A$. Since  $A_K=A_F$,   the greatest lower bound $\widetilde \gamma$ of $ A$
is strictly less than the greatest lower bound $\gamma$
of the symmetric operator $\dot A$. Since  $\dot A$ has  deficiency indices $(1,1)$, $\widetilde \gamma$ is
an eigenvalue of $ A$. Following the strategy of proof of the second statement
of Theorem \ref{tri},  one concludes that the orbit $\Sigma=\bigcup_{g\in G}g(\widetilde \gamma)
$
belongs to the spectrum of the $G$-invariant self-adjoint operator  $ A$.
By Remark \ref{rem53}, the greatest lower bound $\gamma$ is a fixed point for any transformation $g$ from the group $G$,
\begin{equation}\label{hold}
g(\gamma)=\gamma, \quad g\in G,
\end{equation}
that is, $g(x)=a(x-\gamma)+\gamma$ for some $a>0$ (cf. \eqref{g1}).
Since $\widetilde \gamma<\gamma$ and \eqref{hold} holds, it is also clear that  $\inf \Sigma=-\infty$.
Therefore, $ A$ is not semi-bounded from
below, so is $\dot A$.  A contradiction.
\end{proof}


We recall, see Remark \ref{pervaia}, that the fixed point result, Theorem  \ref{Shauder}, does not  guarantee
  the existence of a  self-adjoint unitarily affine  invariant extension of a symmetric invariant operator in general.
However,   a $G$-invariant  symmetric operator with  deficiency indices $(1, 1) $ always has an affine invariant self-adjoint extension  if not  with respect to  the whole group $G$ but
at least
 with respect to a cyclic subgroup of $G $, provided that $G$  is   a continuous one-parameter group.

\begin{theorem}\label{discrete-invariant}
Assume Hypothesis \ref{kiki}.
Suppose that
 $G=\{g_t\}_{t\in \bbR}$   is   a continuous one-parameter group and that $\dot A$ has no self-adjoint  $G$-invariant extension.

Then
  there exists a  discrete cyclic subgroup $G[g]$ generated by some element $g\in G$
  such that  any maximal dissipative, in particular, any self-adjoint  extension of $\dot A$ is $G[g]$-invariant.
 \end{theorem}


\begin{proof}
  Let $V\in \bbT$  be the  von Neumann parameter on the unit circle and $A_V$ the
  associated self-adjoint  extension of  $\dot A$.
  Introduce the continuous flow on the circumference $\bbT$ by
  \begin{equation}\label{defff}
  \varphi(t,V)=\Gamma_{g_t}(V), \quad t\in \bbR, \quad g_t\in G.
  \end{equation}

To prove the assertion, it is sufficient to show that
 the  map  $t\mapsto \varphi(t,V)$ is  not injective for some $V\in \bbT$.

 Indeed, if this map is not injective, then
   there exist two different  values $t_1 $ and $t_2$  of the parameter $t$ such that
\begin{equation}\label{vse}
\varphi(t_1,V)=\varphi(t_2,V).
\end{equation}
Therefore,  $V$ is a fixed point of the transformation
$\varphi(t_2-t_1, \cdot )$. Hence, for any $n\in \bbZ$
$$
V=\varphi(n(t_2-t_1),V),
$$
  and, therefore, the extension $A[g]=A_V$
 is $G[g]$-invariant with respect to the
 discrete subgroup $G[g]$ generated by the element
$$g=g_{t_1-t_2}.$$
By hypothesis the self-adjoint operator $A[g]$ is not $G$-invariant (with respect to the whole subgroup $G$)
 and therefore $V$ is not a fixed point for the whole family of the transformations $\gamma_t$. In particular, the orbit
 $S_V=\cup_{t\in \bbR}\varphi(t,V)$ is an infinite set. Next, if  $W=\varphi(t,V)\in S_V$ for some $t\in \bbR$, then
 $$
 \varphi(T,W)=\varphi(T+t,V)=\varphi(t,V)=W,
 $$
 where
 \begin{equation}\label{period}T=t_2-t_1.
 \end{equation}

 Therefore
$\gamma_{nT}(W)=W$ for all $W\in S_V$ and all $n\in \bbZ$  and hence the corresponding self-adjoint operators $A_W$, $W\in S_V$, are  $G[g]$-invariant.  Thus, since the set $S_V$ is infinite, the operator
 $\dot A$ has infinitely many, and, therefore, at least three  different $G[g]$-invariant self-adjoint extensions.  Therefore,   by Theorem \ref{tri}, every maximal dissipative extension of $\dot A$ (obtained as a restriction of $(\dot A)^*$) is $G[g]$-invariant.

To complete the proof  it remains  to justify the claim that the map  $t\mapsto \varphi(t,V)$  is not injective for some $V\in \bbT$.

To do that assume the opposite. That is, suppose that the  map
  $t\mapsto \varphi(t,V)$ is injective for all $V\in \bbT$.   Then, by Brouwer's invariance of domain theorem \cite{Br},  the set
  $\displaystyle{S_{V}=\bigcup_{t\in \bbR}\varphi(t,V)}$
 is an open arc.\footnote{We are indebted to N. Yu.  Netsvetaev who attracted
  our attention to this fact.} Thus, 
  $$
  \bbT=\bigcup_{V\in \bbT}S_{V}
  $$
  is an open cover. Since $\bbT$ is a compact set, choosing a finite sub cover, 
  $$
  \bbT=\bigcup_{k=1}^NS_{V_k},\quad \text{for some }\quad N\in \bbN, 
  $$
  we get  that $\bbT$ is covered by finitely many  open arcs and therefore some of them must intersect. If so, the circumference $\bbT$ would be 
  the trajectory of some point $V_*\in \bbT$. That is, $\bbT=S_{V_*}$ which means that   the map  $t\mapsto \varphi(t,V_*)$ is not injective.
   A contradiction.

  The proof is complete.
\end{proof}


\begin{remark}\label{unnonself} It follows from Theorems \ref{Shauder} and  \ref{tri}
that under hypothesis of Theorem \ref{discrete-invariant}, that is, if $\dot A$ has no $G$-invariant self-adjoint extensions,  the symmetric operator $\dot A$ admits a unique non-self-adjoint maximal dissipative $G$-invariant extension $A$
with the property $\dot A \subset A\subset  (\dot A)^*$.
\end{remark}

\begin{remark}\label{top} Concrete examples show (see Examples \ref{ex1} and \ref{ex2} and  Subsections \ref{7.1} and \ref{7.5}, respectively)  that our hypothesis that  $\dot A$ has no self-adjoint  $G$-invariant extension is not empty.
Accordingly,  this hypothesis implies that the flow
$\varphi(t, V)=\gamma_t(V)$ on the circumference $\cM=\bbT$ given by \eqref{defff} has no fixed point. It is no surprise: the Euler characteristics of $\bbT$ is zero, $\chi (\bbT)=0$,
and hence Proposition \ref{Lef} cannot be applied. Instead, the flow is periodic with period $T$ given by \eqref{period} and  thus
$$
\varphi(T, \cdot )=\text{Id}_{\bbT},
$$
So, any point $V\in \bbT$ is a fixed point,
$$
\varphi(nT, V)=V,  \quad n\in \bbZ.
$$
\end{remark}

\section{examples}

In this section we provide several examples that illustrate our main results.
In particular, we show that, even in the simplest cases, the restricted standard or generalized Weyl commutation relations   possess a variety of non-unitarily equivalent
solutions.

\begin{example}\label{ex1}
Suppose that $G=\{g_t\}_{t\in \bbR}$ is the one-parameter  group of translations
\begin{equation}\label{shi}
g_t(x)=x+t.
\end{equation}
Let $\dot A$ be the differentiation operator in $L^2(0, \ell)$,
$$
(\dot A) f(x)=i\frac{d}{dx} f(x)\quad \text{on}  \quad \Dom (\dot A) =\{f\in W^1_2(0,\ell)\, |\,  f(0)=f(\ell)=0\}.
$$
It is well know \cite{AkG} that $\dot A$ has deficiency indices $(1,1)$.  Moreover, $\dot A$ is  obviously $G$-invariant  with respect to
the group of unitary transformations  given by
$$
(U_tf)(x)=e^{ixt}f(x), \quad f\in L^2(0,\ell).
$$
\end{example}

\subsection{}\label{7.1} First,  we notice that the symmetric $G$-invariant operator $\dot A$   from Example \ref{ex1} does not admit self-adjoint $G$-invariant extensions.

Indeed, if $A$ is a self-adjoint extension of $\dot A$, then there exists a $\Theta
\in [0, 2\pi)$ such that $A$ coincides with the differentiation operator on
$$
\Dom(A)=\{f\in W^{1}_2(0,\ell)\, | f(\ell)=e^{i\Theta} f(0)\}.
$$
Therefore, as a simple computation shows,
the spectrum of the operator $A$
 is discrete and it consists of simple eigenvalues  forming the  arithmetic progression
\begin{equation}\label{latt}
\spec (A)=\bigcup_{n\in \bbZ}\left \{\frac{\Theta +2\pi n}{\ell}\right \}.
\end{equation}
However, the spectrum of a $G$-invariant operator is a $G$-invariant set:
if $\lambda $
is an eigenvalue of  $A$, then for all $t\in \bbR$ the point  $g_t(\lambda)=\lambda+t$ is also  an  eigenvalue of $A$   which is impossible, for the underlying Hilbert space is  separable.

Thus, $  \dot A$  has no  $G$-invariant self-adjoint extension.

This observation  justifies  the claim  from  Remark \ref{pervaia}
that the restricted flow $\Gamma_g|_{\cU}$  may not have fixed points at all, even in the case of equal deficiency indices.

\subsection{}Next, since $\dot A$ from Example \ref{ex1} has no $G$-invariant self-adjoint extension,  our general fixed point Theorem \ref{Shauder} states   that there exists a maximal dissipative
 $G$-invariant extension. In turn, Theorem \ref{tri} shows then that this extension is unique.
 Clearly,  that $G$-invariant extension is given by the differentiation operator  with the Dirichlet boundary condition at the left end-point of the interval,
\begin{equation}\label{disdom}
(Af)(x)=i\frac{d}{dx}f(x) \quad \text{on} \quad \Dom(A)=\{f\in W_2^1(0, \ell)\, |\, f(0)=0\}.
\end{equation}

\subsection{}\label{can}
It is worth mentioning  that the dissipative  extension $A$ defined  by \eqref{disdom} generates a nilpotent semi-group of shifts
$V_s$
 with index $\ell$. Indeed,
$$(V_sf)(x)=\begin{cases}f(x-s),&x>s\\
0,&x\le s\end{cases},
\quad f\in
L^2(0,\ell),
$$
and therefore
$V_s=0$ for $ s\ge \ell.
$

An  immediate computation shows that the restricted  Weyl commutation relations
\begin{equation}\label{nunu}
U_tV_s=e^{ist}V_{s}U_t, \quad t\in \bbR, \quad s\ge0,
\end{equation}
for the unitary group $U_t$ and the semi-group of contractions $V_s$ are satisfied.

We remark that the generators $A$, the differentiation operators on a finite interval defined by \eqref{disdom},
 are not unitarily equivalent for different values of the parameter $\ell$, the length  of the interval
 $[0, \ell]$ (see \cite{AkG}). Therefore, the restricted Weyl
commutation relations
 \eqref{nunu}
 admit a variety of non-unitarily equivalent solutions.

\subsection{}Finally, we illustrate Theorem \ref{discrete-invariant}.

One observes that   any maximal dissipative, in particular, self-adjoint extension of $\dot A$
has the property that
\begin{equation}\label{inv}
U_t\left ( \Dom (A)\right )=\Dom (A)\quad \text{for all }\quad t=\frac{2\pi }{\ell}n, \quad n \in \bbZ.
\end{equation}
 Therefore, $A$ is $G[g]$-invariant, where  $G[g]$ is the discrete subgroup
 of the one-parameter  group $G$ \eqref{shi} generated by the translations
 $$
 g(x)=x+\frac{2\pi}{\ell}.
 $$

 Indeed, as it is shown in \cite{AkG}, the domain of any maximal dissipative extensions admits the representation
 $$\Dom (A)= \{f\in W^{1}_2(0,\ell)\, | f(0)=\rho f(\ell)\}$$
  for some $|\rho|\le 1$. Therefore,  the domain invariance property \eqref{inv} automatically  holds. In particular, the spectrum of $A$
  is a lattice in the upper half-plane with  period
 $ \frac{2\pi}{\ell}$,  whenever  $A$ is maximal dissipative non-self-adjoint.
 If $|\rho|=1$,  and hence  the operator $A$ is self-adjoint, the spectrum coincides with the set of points forming 
 the arithmetic progression \eqref{latt} (for some $\Theta$).

Our second example relates to the homogeneous Schr\"odinger operator with a singular potential. In this example,
   the underlying one-parameter group  coincides with the group of  scaling transformations of the real axis.

\begin{example}\label{ex2}
Suppose that $G=\{g_t\}_{t\in \bbR}$ is the group of scaling   transformations
$$
g_t(x)=e^{t}x.
$$
Let $\dot A$ be
the closure in $L^2(0,\infty)$ of the
symmetric operator   given by the differential expression
\begin{equation}\label{formself}
\tau_\nu=-\frac{d^2}{dx^2}+\frac{\gamma}{x^2}
\end{equation}
initially defined on $C_0^\infty(0, \infty)$.

It is known, see, e.g., \cite[Ch. X]{Simon}, that

\begin{itemize}
 \item[(i)] if   $\gamma <\frac34$, the operator $\dot A$ has  deficiency indices $(1,1)$,
\item[] and
\item[(ii)]
 it
 is non-negative  for  $-\frac14\le \gamma $,
 \item[]
  while
  \item[(iii)]  $\dot A $ is not semi-bounded from below for $\gamma<-\frac14$.
  \end{itemize}
  
Regardless of the magnitude  of the coupling constant $\gamma$, the operator $\dot A$ is $G$-invariant (homogeneous) with respect to the unitary  group of scaling transformations
\begin{equation}\label{give9}
(U_tf)(x)=e^{\frac{t}{4}}f(e^{\frac{t}{2}}x), \quad f\in L^2(\bbR_+).
\end{equation}

\end{example}

\subsection{}\label{7.5} First,  we notice that as in Example \ref{ex1},  the symmetric $G$-invariant  operator $\dot A$  from Example \ref{ex2} with
$$
\gamma<-\frac14
$$
does not admit self-adjoint $G$-invariant extensions at all.

Indeed,
 it is well known (see, e.g., \cite{PZ}) that if  $A$ is a self-adjoint extension of $\dot A$, then the  domain of $\dot A$  can be characterized by the following asymptotic
 boundary condition: there exists a $\Theta \in [0, 2\pi )$ such that for any
 $$
 f\in  \Dom (A)\subset  W^{2}_{2, \text{loc}}(0,\infty)  $$
 the asymptotic representation
 \begin{equation}\label{asbound}
 f(x)\sim C\sqrt{x}\sin\left ( \nu \log  x +\Theta\right ) \quad \text{as}\,\, x\downarrow 0,
 \end{equation}
holds for some $C\in \bbC$, where
$$
\nu= \sqrt{\left |\gamma+\frac14\right |}.
$$

Clearly,
$$
U_t \left (\Dom (A)\right )=\Dom (A) \quad \text{for all }\quad t=\frac{4\pi n}{\nu}, \quad n \in \bbZ,
$$
where the scaling group $U_t$ is given by \eqref{give9}. Therefore,  the extension $ A$ is  $G[g]$-invariant where
 $G[g]$ denotes the discrete cyclic subgroup generated by the linear transformation $g\in \cG$,
 \begin{equation}\label{general}
 g(x)=\kappa x \quad \text{for } \quad  \kappa = e^{4\pi/\nu}.
 \end{equation}
 In particular, if $\lambda_0$ is a negative eigenvalue of the self-adjoint extension $A$, then the two-sided geometric progression
 \begin{equation}\label{2side}
 \lambda_n=\kappa^n \lambda_0, \quad n \in \bbZ,
 \end{equation}
  fills in the (negative) discrete spectrum of A.

  It remains to argue exactly as in the case of the differentiation operator on a finite interval  from Example \ref{ex1}
to show that $\dot A$ has no $G$-invariant self-adjoint extension.
Also, applying Theorem \ref{discrete-invariant} shows that any maximal extension of $\dot A$ is invariant with respect to the cyclic group $G[g]$, where $g$ is given by \eqref{general}.

\begin{remark}\label{fall} We remark  that the self-adjoint realizations of the Schr\"odinger operator
$$
H=-\frac{d^2}{dx^2}+\frac{\gamma}{x^2}, \quad \gamma<-\frac14,
$$
 defined  by the asymptotic boundary conditions \eqref{asbound}, are well-suited for modeling the fall to the center phenomenon in quantum mechanics \cite[Sec. 35]{LL}, see also \cite{PP} and \cite{PZ}.

From the group-theoretic standpoint this phenomenon can be expressed  as follows. The  initial symmetric homogeneous operator (with respect to the scaling transformations) is neither semi-bounded from below nor admits self-adjoint homogeneous realizations, while any self-adjoint realization of the corresponding Hamiltonian is homogeneous with respect to a cyclic subgroup of the   transformations only.
As a consequence, its negative discrete spectrum as a set  is invariant with respect to the natural action of this  subgroup of  scaling transformations of the real axis.
That is,
$$
\spec(A)\cap (-\infty, 0)=\bigcup_{n\in \bbZ}\{\kappa^n\lambda_0\}.
$$

\end{remark}

\subsection{}\label{8.6}  Theorem \ref{Shauder} combined with  Theorem \ref{tri} shows that
$\dot A$
has a unique maximal dissipative $G$-invariant  extension which coincides with    the following
 homogeneous operator  $A$,   the restriction of the adjoint operator
$(\dot A)^*$ on
\begin{equation}\label{domai}
\Dom(A)=\left \{f\in \Dom ((\dot A)^* )\, \bigg |\,
\lim_{x\downarrow 0}\frac{f(x)}{x^{1/2+i\sqrt{\left |\gamma+\frac14\right |} }}\,  \text{ exists}\right \}.
\end{equation}
It is clear that $A$
is $G$-invariant. We claim without  proof that    $A$ is a maximal dissipative operator.

 \subsection{}\label{gencan} The dissipative Schr\"odingier operator   $A$  with a singular potential from subsection \ref{8.6} generates
 the contractive  semi-group $ V_s=e^{iAs}$, $s\ge0$.
 In accordance with  Theorem  \ref{bccr2},  the restricted generalized Weyl commutation relations  \begin{equation}\label{ccrcr}
 U_tV_s=\exp (ise^{t})V_{e^{t}s}U_t, \quad t\in \bbR, \,\,\, s\ge0,
 \end{equation}
 hold.
 One can also show that the homogeneous extensions $A$ are not unitarily equivalent for different values of the coupling constant $\gamma $ satisfying the inequality  $\gamma<-\frac14$.
Therefore, the restricted  generalized Weyl commutation relations \eqref{ccrcr} admit
 a continuous family of non-unitarily equivalent solutions (cf. Subsection \ref{can}).

\subsection{}  Finally,   we   illustrate Theorem    \ref{FKnet} which deals with the case of semi-bounded $G$-invariant operators.

Assuming that
\begin{equation}\label{buss}
-\frac14\le \gamma <\frac34,
\end{equation}  one observes that
 the symmetric operator $\dot A$ is non-negative. Moreover, it is clear that the following two homogenous extensions,
 the restrictions of the adjoint operator
$(\dot A)^*$ on
$$
\Dom(A_F)=\left \{f\in \Dom ((\dot A)^* )\, \bigg |\,
\lim_{x\downarrow 0}\frac{f(x)}{x^{1/2+\sqrt{\gamma+\frac14} }}\,  \text{ exists}\right \}
$$
and  on
$$
\Dom(A_K)=\left \{f\in \Dom ((\dot A)^* )\, \bigg |\,
\lim_{x\downarrow 0}\frac{f(x)}{x^{1/2-\sqrt{\gamma+\frac14} }}\,  \text{ exists}\right \},
$$
respectively,
are $G$-invariant with respect to the unitary representation \eqref{give9}.
We remark that the self-adjoint realization $A_F$ and $A_K$   (cf.  \cite{EK,GP,Ka}) can be recognized  as
 the Friedrichs and Krein-von Neumann extensions of $\dot A$, respectively,   which agrees   with the statement of Theorem \ref{FKnet}.

We also remark that under hypothesis \eqref{buss},   the corresponding flow has two different fixed point on the unit circle.
In the critical case of
$$\gamma=-\frac14,$$ the  extensions  $A_K$ and $A_F$  coincide. Hence, the flow \eqref{flow1} has  only one fixed point on the boundary $\bbT$ of the unit disk $\bbD$ (cf. Proposition \ref{nado}).

We conclude this section by an example of a $G$-invariant generator that has deficiency indices $(0,1)$.
 We will see below that this example is universal in the following sense. The semi-group of unilateral shifts $V_s=e^{isA}$, $s\ge 0$, generated by $A$ solves the restricted  standard as well as  generalized Weyl commutation relations  \eqref{ninii} and \eqref{nini2}.

\begin{example}\label{ex3}
Let $\dot A$ be the Dirichlet differentiation operator on the positive semi-axis  (in $L^2(0, \infty)$),
$$
(\dot A) f(x)=i\frac{d}{dx} f(x) \quad \text{on} \quad \Dom (\dot A) =\{f\in W^1_2(0,\infty)\,|\, f(0)=0\}.
$$
It is well known, see, e.g.,   \cite{AkG},  that the symmetric operator $\dot A$  has  deficiency indices $(0,1)$ and it  is  simultaneously   a maximal dissipative  operator.

Suppose that $G^{(1)}=\{g_t\}_{t\in \bbR}$ is the group of affine transformations
$$
g_t(x)=x+t.
$$
Clearly,  $\dot A$ is $G^{(1)}$-invariant  with respect to
the group of unitary transformations  given by
$$
(U_t^{(1)}f)(x)=e^{ixt}f(x), \quad f\in L^2(0,\infty).
$$

Denoting by   $G^{(2)}=\{g_t\}_{t\in \bbR}$  the group of scaling transformations
$$
g_t(x)=e^tx,
$$
one also concludes that  $\dot A$ is $G^{(2)}$-invariant  with respect to
the group of unitary transformations  given by
\begin{equation}\label{give10}
(U_t^{(2)}f)(x)=e^{\frac{t}{2}}f(e^{t}x), \quad f\in L^2(0,\infty).
\end{equation}
In particular, for the semi-group  of unilateral right shifts  $V_s=e^{isA}$  generated by the differentiation operator    $A=\dot A$, one gets
the restricted Weyl commutation relations
$$
U_t^{(1)}V_s=e^{ist}V_{s}U_t^{(1)}, \quad t\in \bbR, \quad s\ge0,
$$
as well as the  restricted generatlized  Weyl commutation relations
$$
U_t^{(2)}V_s=\exp (ise^{t})V_{e^{t}s}U_t^{(2)}, \quad t\in \bbR, \,\,\, s\ge0.
$$

We also remark that the operator $A=\dot A$ is $\cG$-invariant with respect to the unitary representation
$\cG\ni g\mapsto U_g$  of the whole affine group $\cG$ given by
$$
(U_g)f(x)=a^{1/2}e^{ibx}f(ax),\quad f\in L^2(0,\infty),  \quad g\in \cG, \quad g(x)=ax+b.
$$
\end{example}

\end{document}